\documentclass[12pt, reqno]{amsart}
\usepackage{amsmath, amsthm, amscd, amsfonts, amssymb, graphicx, color}
\usepackage[bookmarksnumbered, colorlinks, plainpages]{hyperref}

\newtheorem{theorem}{Theorem}[section]
\newtheorem{lemma}[theorem]{Lemma}
\newtheorem{proposition}[theorem]{Proposition}
\newtheorem{corollary}[theorem]{Corollary}
\theoremstyle{definition}

\theoremstyle{remark}
\newtheorem{remark}[theorem]{Remark}
\numberwithin{equation}{section}

\begin{document}
\setcounter{page}{1}

\title[Inequalities involving Heinz and Heron Functional Means]
{Some Inequalities involving Heron and Heinz Means of two Convex Functionals}
\author[M. Ra\"{\i}ssouli, S. Furuichi]{Mustapha Ra\"{\i}ssouli$^{1,2}$ and Shigeru Furuichi$^{3}$}
\address{$^1$ Department of Mathematics, Science Faculty, Taibah University,
Al Madinah Al Munawwarah, P.O.Box 30097, Zip Code 41477,
Kingdom of Saudi Arabia.}

\address{$^2$ Department of Mathematics, Faculty of Science, Moulay Ismail University, Meknes, Morocco.}

\address{$^3$ Department of Information Science, College of Humanities and Sciences, Nihon University, 3-25-40, Sakurajyousui, Setagaya-ku, Tokyo, 156-8550, Japan.}

\email{\textcolor[rgb]{0.00,0.00,0.84}{raissouli.mustapha@gmail.com}}
\email{\textcolor[rgb]{0.00,0.00,0.84}{furuichi@chs.nihon-u.ac.jp}}



\date{Received: xxxxxx; Revised: yyyyyy; Accepted: zzzzzz.}

\begin{abstract}
In this paper we first introduce the Heron and Heinz means of two convex functionals. Afterwards, some inequalities involving these functional means are investigated. The operator versions of our theoretical functional results are immediately deduced. We also obtain new refinements of some known operator inequalities via our functional approach in a fast and nice way.
\end{abstract}

\maketitle

\section{\bf Heron and Heinz Means}

Let $a,b>0$ and $\lambda\in[0,1]$ be real numbers. The following expressions
\begin{equation}\label{eq1}
a\nabla_{\lambda}b=(1-\lambda)a+\lambda b,\;\; a\sharp_{\lambda}b=a^{1-\lambda}b^{\lambda}
\end{equation}
are known as the $\lambda$-weighted arithmetic mean and $\lambda$-weighted geometric mean, respectively. They satisfy the following
\begin{equation}\label{2}
a\sharp_{\lambda}b\leq a\nabla_{\lambda}b
\end{equation}
known as the Young's inequality. Some refinements and reverses of \eqref{2} have been discussed in the literature. In particular, the following result has been proved in \cite{KM},
\begin{equation}\label{2a}
r_{\lambda}\big(\sqrt{a}-\sqrt{b}\big)^2\leq a\nabla_{\lambda}b-a\sharp_{\lambda}b \leq (1-r_{\lambda})\big(\sqrt{a}-\sqrt{b}\big)^2,
\end{equation}
where we set
\begin{equation}\label{2b}
r_{\lambda}:=\min(\lambda,1-\lambda).
\end{equation}

For more refinements and reverses of the Young's inequality, we refer the interested reader to \cite{F,M1,T} and the related references cited therein.

From the previous means we introduce the following expressions
\begin{equation}\label{3a}
K_{\lambda}(a,b)=(1-\lambda)\sqrt{ab}+\lambda\frac{a+b}{2},
\end{equation}
\begin{equation}\label{3b}
 HZ_{\lambda}(a,b)=\frac{a^{1-\lambda}b^{\lambda}+a^{\lambda}b^{1-\lambda}}{2}
\end{equation}
known in the literature as the Heron and Heinz means, respectively. They satisfy the following inequalities
\begin{equation}\label{4a}
\sqrt{ab}\leq K_{\lambda}(a,b)\leq\frac{a+b}{2},
\end{equation}
\begin{equation}\label{4b}
\sqrt{ab}\leq HZ_{\lambda}(a,b)\leq\frac{a+b}{2}.
\end{equation}

An inequality between the Heron and Heinz means was proved in \cite{B} and is as follows
\begin{equation}\label{5}
HZ_{\lambda}(a,b)\leq K_{\alpha(\lambda)}(a,b),
\end{equation}
where $\alpha(\lambda)=(2\lambda-1)^2$ for any $\lambda\in[0,1]$. It is easy to check that \eqref{5} is better than the following inequality
$$HZ_{\lambda}(a,b)+r_{\lambda}\big(\sqrt{a}-\sqrt{b}\big)^2\leq\frac{a+b}{2}$$
which has been later obtained in \cite{KM}, where $r_{\lambda}$ is defined by \eqref{2b}.

A reverse of Heinz inequality was recently proved in \cite{KSS} as follows
\begin{equation}\label{5a}
HZ_{\lambda}(a,b)\geq\frac{a+b}{2}-\frac{1}{2}\lambda(1-\lambda)(b-a)\log(b/a).
\end{equation}
Another reversed version of Heinz inequality was already showed in \cite{KM2} and reads as follows
\begin{equation}\label{5b}
\Big(HZ_{\lambda}(a,b)\Big)^2\geq\Big(\frac{a+b}{2}\Big)^2-\frac{1}{2}(1-r_{\lambda})(a-b)^2.
\end{equation}

In \cite{KSS}, the authors mentioned some comments about comparison between \eqref{5a} and \eqref{5b} that we present in the following remark (see page 745).

\begin{remark}
Numerical experiments show that neither \eqref{5a} nor \eqref{5b} is uniformly better than the
other. However, these experiments show that, for most values of $\lambda$ , \eqref{5a} is better than
\eqref{5b} when $a/b$ is relatively small and \eqref{5b} is better when $a/b$ is large.
\end{remark}

The extension of the previous means, from the case that the variables are positive real numbers to the case that the arguments are positive operators, has been investigated in the literature. Let $H$ be a complex Hilbert space and ${\mathcal B}(H)$ be the $\mathbb{C}^*$-algebra of bounded linear operators acting on $H$. We denote by ${\mathcal B}^{+*}(H)$ the open cone of all (self-adjoint) positive invertible operators in ${\mathcal B}(H)$. For $A,B\in {\mathcal B}^{+*}(H)$, the following expressions
\begin{equation*}
A\nabla_{\lambda}B:=(1-\lambda)A+\lambda B=B\nabla_{1-\lambda} A,
\end{equation*}
\begin{equation*}
A\sharp_{\lambda}B:=A^{1/2}\Big(A^{-1/2}BA^{-1/2}\Big)^{\lambda}A^{1/2}=B\sharp_{1-\lambda}A
\end{equation*}
are known as the $\lambda$-weighted operator mean and $\lambda$-weighted geometric operator mean of $A$ and $B$, respectively. For $\lambda=1/2$, they are simply denoted by $A\nabla B$ and $A\sharp B$, respectively. These operator means satisfy the following inequality
\begin{equation}\label{6}
A\sharp_{\lambda}B\leq A\nabla_{\lambda}B,
\end{equation}
which is an operator version of the Young's inequality \eqref{2}. The notation $\leq$ refers here for the L\"{o}wner partial order defined by: $T\leq S$ if and only if $T$ and $S$ are self-adjoint and $S-T$ is positive.

An operator version of \eqref{2a} has also been established in \cite{KM} and reads as follows
\begin{equation}\label{65}
2r_{\lambda}\big(A\nabla B-A\sharp B\big) \leq A\nabla_{\lambda}B -A\sharp_{\lambda}B \leq2(1-r_{\lambda})\big(A\nabla B-A\sharp B\big),
\end{equation}
where $r_{\lambda}$ is defined in \eqref{2b}. In fact, according to the Kubo-Ando theory \cite{KA}, \eqref{65} can be immediately deduced from \eqref{2a}.

By analogy with the scalar case, the Heron and Heinz operator means are, respectively, defined as follows
\begin{equation}\label{7a}
K_{\lambda}(A,B)=(1-\lambda)A\sharp B+\lambda A\nabla B,
\end{equation}
\begin{equation}\label{7b}
HZ_{\lambda}(A,B)=\frac{A\sharp_{\lambda}B+A\sharp_{1-\lambda}B}{2}.
\end{equation}

The following operator inequalities, extending respectively \eqref{4a} and \eqref{4b} for operator arguments, have also been proved in the literature, see \cite{YR} and the related references cited therein.
\begin{equation}\label{8a}
A\sharp B\leq K_{\lambda}(A,B)\leq A\nabla B,
\end{equation}
\begin{equation}\label{8b}
A\sharp B\leq HZ_{\lambda}(A,B)\leq A\nabla B.
\end{equation}

The following refinement of the inequality in \eqref{8a} has been recently obtained in \cite{YR}
\begin{multline}\label{9}
\lambda(1-\lambda)\big(A\nabla B-A\sharp B\big)+A\sharp B\leq K_{\lambda}(A,B)\\
\leq A\nabla B-\lambda(1-\lambda)\big(A\nabla B-A\sharp B\big).
\end{multline}

For more inequalities related to the Heron and Heinz means involving matrix and operator arguments, we refer the reader to \cite{DDF,FFN, I2,Kho, KCMSS, KM, KM2, KSS, LS, Z} and the related references cited therein.

\section{\bf Functional Version}

The previous operator means have been extended from the case that the variables are positive operators to the case that the variables are convex functionals, see \cite{RB}. Let us denote by $\Gamma_0(H)$ the cone of all $f:H\longrightarrow{\mathbb R}\cup\{+\infty\}$ which are convex lower semi-continuous and not identically equal to $+\infty$. Throughout this paper, we use the following notation:
$${\mathcal D}(H)=\Big\{(f,g)\in \Gamma_0(H)\times \Gamma_0(H) :\;\;{\rm dom}\;f\cap{\rm dom}\;g\neq\emptyset\Big\},$$
where ${\rm dom}\;f$ refers to the effective domain of $f:H\longrightarrow{\mathbb R}\cup\{+\infty\}$ defined by
$${\rm dom}\;f=\Big\{x\in H,\;\; f(x)<+\infty\Big\}.$$

If $(f,g)\in {\mathcal D}(H)$ and $\lambda\in(0,1)$, then the following expressions
\begin{equation}\label{10a}
{\mathcal A}_{\lambda}(f,g):=(1-\lambda)f+\lambda g,
\end{equation}
\begin{equation}\label{10b}
{\mathcal
G}_{\lambda}(f,g):=\displaystyle{\frac{\sin(\pi\lambda)}{\pi}\int_{0}^1\frac{t^{\lambda-1}}{(1-t)^{\lambda}}}\Big((1-t)f^*+tg^*\Big)^*dt
\end{equation}
are called, by analogy, the $\lambda$-weighted functional arithmetic mean and $\lambda$-weighted functional geometric mean of $f$ and $g$, respectively. For $\lambda=1/2$, they are simply denoted by ${\mathcal A}(f,g)$ and ${\mathcal G}(f,g)$, respectively. Here, the notation $f^*$ refers to the Fenchel conjugate of any $f:H\longrightarrow{\mathbb R}\cup\{+\infty\}$ defined through
\begin{equation}\label{13}
\forall x^*\in H\;\;\;\;\;\;\; f^*(x^*)=\sup_{x\in H}\Big\{\Re e\langle x^*,x\rangle-f(x)\Big\}.
\end{equation}

It is easy to see that
\begin{equation}\label{12}
{\mathcal
G}_{\lambda}(f,g)=\displaystyle{\frac{\sin(\pi\lambda)}{\pi}\int_{0}^1\frac{1}{t\sharp_{\lambda}(1-t)}}{\mathcal H}_t(f,g)dt,
\end{equation}
where
\begin{equation}\label{125}
{\mathcal H}_{\lambda}(f,g):=\Big((1-\lambda)f^*+\lambda g^*\Big)^*
\end{equation}
is the so-called $\lambda$-weighted functional harmonic mean of $f$ and $g$. For $\lambda=1/2$, we simply denote it by ${\mathcal H}(f,g)$.

With this, we can write
$$\forall x^*\in H\;\;\;\;\;\;\; f^*(x^*)=\sup_{x\in{\rm dom}\;f}\Big\{\Re e\langle x^*,x\rangle-f(x)\Big\},$$
provided that ${\rm dom}\;f\neq\emptyset$. As supremum of a family of affine (so convex) functions, $f^*$ is always convex even when $f$ is not.

We extend the previous functional means on the whole interval $[0,1]$ by setting, see \cite{RF}
\begin{equation}\label{eq133}
{\mathcal A}_{0}(f,g)={\mathcal G}_{0}(f,g)={\mathcal H}_{0}(f,g)=f,\;\;{\mathcal A}_{1}(f,g)={\mathcal G}_{1}(f,g)={\mathcal H}_{1}(f,g)=g.
\end{equation}

The previous functional means satisfy the following relationships
\begin{equation}\label{135}
{\mathcal A}_{\lambda}(f,g)={\mathcal A}_{1-\lambda}(g,f),\; {\mathcal H}_{\lambda}(f,g)={\mathcal H}_{1-\lambda}(g,f),\;
{\mathcal G}_{\lambda}(f,g)={\mathcal G}_{1-\lambda}(g,f),
\end{equation}
for any $(f,g)\in{\mathcal D}(H)$ and $\lambda\in[0,1]$. The two first relationships are immediate while the proof of the third one can be found in \cite{RB}. In particular, if $\lambda=1/2$, the three previous functional means are symmetric in $f$ and $g$.

We have the following double inequality
\begin{equation}\label{14}
{\mathcal H}_{\lambda}(f,g)\leq{\mathcal
G}_{\lambda}(f,g)\leq {\mathcal A}_{\lambda}(f,g),
\end{equation}
whose the right inequality is the functional version of the Young's operator inequality \eqref{6}. Here the symbol $\leq$ denotes the point-wise order defined by, $f\leq g$ if and only if $f(x)\leq g(x)$ for all $x\in H$.

\begin{remark}\label{rem1}
We adopt here the conventions $0\cdot(+\infty)=+\infty$ and $(+\infty)-(+\infty)=+\infty$, as usual in convex analysis. Since our involved functionals $f$ and/or $g$ can take the value $+\infty$ we then mention the following:\\
(i) The relations \eqref{eq133} are not immediate from their related functional means \eqref{10a}, \eqref{10b} and \eqref{125}, respectively.\\
(ii) We must be careful with any proof of functional equality or inequality. As example, the equalities $f-f=0$ and $f-g=-(g-f)$ are not always true. Also, the two inequalities $f\leq g$ and $f-g\leq0$ are not always equivalent whereas $f\leq g$ and $g-f\geq0$ are equivalent.
\end{remark}

The previous functional means are, respectively, extensions of their related operator means in the following sense
\begin{equation}\label{145}
{\mathcal A}_{\lambda}(f_A,f_B)=f_{A\nabla_\lambda B},\; {\mathcal
G}_{\lambda}(f_A,f_B)=f_{A\sharp_{\lambda}B},\; {\mathcal H}_{\lambda}(f_A,f_B)=f_{A!_{\lambda}B},
\end{equation}
where
$$A!_{\lambda}B:=\big((1-\lambda)A^{-1}+\lambda B^{-1}\big)^{-1}$$
stands for the $\lambda$-weighted harmonic operator mean of $A$ and $B$ and the notation $f_A$ refers to the convex quadratic form generated by the positive operator $A$, i.e. $f_A(x)=(1/2)\langle Ax,x\rangle$ for all $x\in H$. This because $f_A^*(x^*)=(1/2)\langle A^{-1}x^*,x^*\rangle$, or in short $f_A^*=f_{A^{-1}}$, for any $A\in{\mathcal B}^{+*}(H)$. We also mention that, since $A$ and $B$ are self-adjoint then, $f_A=f_B$ if and only if $A=B$.\\

For all $(f,g)\in{\mathcal D}(H)$ and $\lambda\in[0,1]$, we introduce the following expressions
\begin{equation}\label{15}
{\mathcal K}_{\lambda}(f,g)=(1-\lambda){\mathcal G}(f,g)+\lambda{\mathcal A}(f,g)
\end{equation}
and
\begin{equation}\label{16}
{\mathcal{HZ}}_{\lambda}(f,g)=\frac{1}{2}\Big({\mathcal
G}_{\lambda}(f,g)+{\mathcal
G}_{1-\lambda}(f,g)\Big)
\end{equation}
which will be called the Heron functional mean and the Heinz functional mean of $f$ and $g$, respectively. It is clear that ${\mathcal K}_{\lambda}(f,g)$ is symmetric in $f$ and $g$ and, due to the last relation of \eqref{135}, ${\mathcal{HZ}}_{\lambda}(f,g)$ is also symmetric in $f$ and $g$. We mention that we have ${\mathcal{HZ}}_{\lambda}(f,g)={\mathcal{HZ}}_{1-\lambda}(f,g)$ while in general ${\mathcal K}_{\lambda}(f,g)\neq{\mathcal K}_{1-\lambda}(f,g)$, unless $\lambda=1/2$.

By virtue of \eqref{145}, it is not hard to see that, for any $A,B\in{\mathcal B}^{+*}(H)$, we have
\begin{equation}\label{17}
{\mathcal K}_{\lambda}(f_A,f_B)=f_{K_{\lambda}(A,B)},\;\; {\mathcal{HZ}}_{\lambda}(f_A,f_B)=f_{HZ_{\lambda}(A,B)}.
\end{equation}
Furthermore, according to \eqref{14} for $\lambda=1/2$ we immediately deduce that the functional map $\lambda\longmapsto{\mathcal K}_{\lambda}(f,g)$, for fixed $(f,g)\in{\mathcal D}(H)$, is point-wisely increasing in $\lambda\in[0,1]$. This gives the functional version of the operator inequality \eqref{8a} that reads as follows
\begin{equation}\label{18}
{\mathcal G}(f,g)\leq {\mathcal K}_{\lambda}(f,g)\leq{\mathcal A}(f,g),
\end{equation}
which, in its turn, immediately yields \eqref{8a} by virtue of \eqref{145} and \eqref{17}.

Now, a question arises from the above: Is the functional version of \eqref{8b} true when the operator variables $A$ and $B$ are replaced by convex functionals. Precisely, is the following
\begin{equation}\label{19}
{\mathcal G}(f,g)\leq {\mathcal{HZ}}_{\lambda}(f,g)\leq{\mathcal A}(f,g)
\end{equation}
hold for any $(f,g)\in{\mathcal D}(H)$ and $\lambda\in[0,1]$. An affirmative answer to this latter question will be discussed in the next section.

\begin{remark}
Usually, for proving an operator inequality like \eqref{6}, \eqref{65}, \eqref{8a}, \eqref{8b} and \eqref{9} we start from its analog for scalar case and we then proceed by using the techniques of functional calculus. In this paper, after defining the Heron and Heinz means of two convex functionals, we establish some inequalities involving these functional means. We also obtain new refinements of some known operator inequalities. Our approach is with functional character and the proofs of our theoretical results are short, simple and nice and do not need to use the techniques of functional calculus.
\end{remark}

\section{\bf The Main Results}

We preserve the same notations as in the previous sections. Our first main result is about a refinement of \eqref{18} recited in the following.

\begin{theorem}\label{th1}
For any $(f,g)\in{\mathcal D}(H)$ and $\lambda\in[0,1]$ there hold
\begin{multline}\label{20}
r\Big({\mathcal A}(f,g)-{\mathcal
G}(f,g)\Big)+{\mathcal
G}(f,g)={\mathcal K}_{r}(f,g)\leq{\mathcal K}_{\lambda}(f,g)\\
\leq{\mathcal K}_{1-r}(f,g)={\mathcal A}(f,g)-r\Big({\mathcal A}(f,g)-{\mathcal
G}(f,g)\Big),
\end{multline}
where we set $r=r_{\lambda}:=\min(\lambda,1-\lambda)$ for the sake of simplicity.
\end{theorem}
\begin{proof}
Since $\lambda\longmapsto {\mathcal K}_{\lambda}(f,g)$, for fixed $(f,g)\in{\mathcal D}(H)$, is point-wisely increasing in $\lambda \in[0,1]$ and $r_{\lambda}\leq\lambda\leq R_{\lambda}=1-r_{\lambda}$ we then immediately deduce the desired inequalities.
\end{proof}

The operator version of Theorem \ref{th1} reads as follows.

\begin{corollary}\label{cor1}
The following operator inequalities
\begin{equation}\label{21}
r_{\lambda}\Big(A\nabla B-A\sharp B\Big)+A\sharp B\leq K_{\lambda}(A,B)\leq A\nabla B-r_{\lambda}\Big(A\nabla B-A\sharp B\Big)
\end{equation}
hold for any $A,B\in{\mathcal B}^{+*}(H)$ and $\lambda\in[0,1]$.
\end{corollary}
\begin{proof}
It is immediate from \eqref{20} when we take $f=f_A$ and $g=f_B$ and we use \eqref{145} and \eqref{17}. Details are simple and therefore omitted here.
\end{proof}

\begin{remark}
(i) It is clear that \eqref{21} refines the double inequality \eqref{8a}. Further, \eqref{21} also refines \eqref{9} since $\lambda(1-\lambda)\leq r_{\lambda}:=\min(\lambda,1-\lambda)$ for any $\lambda\in[0,1]$. Note that \eqref{9} has been proved in \cite{YR} via the techniques of functional calculus.\\
(ii) The operator inequality \eqref{21} has been immediately deduced from \eqref{20} which, in its turn, refines \eqref{18}. Moreover, \eqref{20} has a simple proof and immediately implies \eqref{21} without the need to the techniques of functional calculus.
\end{remark}

Before stating our second main result we need the following lemma.

\begin{lemma}\label{lem1}
Let $(f,g)\in{\mathcal D}(H)$ and $\lambda\in(0,1)$. Then we have
\begin{equation}\label{22}
{\mathcal A}_{\lambda}(f,g)-{\mathcal G}_{\lambda}(f,g)=\frac{\sin(\pi\lambda)}{\pi}\int_0^1\frac{1}{t\sharp_{\lambda}(1-t)}\Big({\mathcal A}_t(f,g)-{\mathcal H}_t(f,g)\Big)dt.
\end{equation}
In particular, one has
\begin{equation}\label{225}
{\mathcal A}(f,g)-{\mathcal G}(f,g)=\frac{1}{\pi}\int_0^1\frac{1}{\sqrt{t(1-t)}}\Big({\mathcal A}_t(f,g)-{\mathcal H}_t(f,g)\Big)dt.
\end{equation}
\end{lemma}
\begin{proof}
Let $\Gamma$ and $B$ denote the standard special functions Gamma and Beta, respectively. Then we have
$$\frac{\sin(\pi\lambda)}{\pi}\int_0^1\frac{t^{\lambda-1}}{(1-t)^{\lambda}}dt=\frac{\sin(\pi\lambda)}{\pi}B\big(\lambda,1-\lambda\big)
=\frac{\sin(\pi\lambda)}{\pi}\Gamma(\lambda)\Gamma(1-\lambda)=1.$$
This, with the definition of ${\mathcal G}_{\lambda}(f,g)$, yields
$${\mathcal A}_{\lambda}(f,g)-{\mathcal G}_{\lambda}(f,g)=\frac{\sin(\pi\lambda)}{\pi}\int_0^1\frac{t^{\lambda-1}}{(1-t)^{\lambda}}
\Big({\mathcal A}_{\lambda}(f,g)-{\mathcal H}_{t}(f,g)\Big)dt.$$
Now, since ${\mathcal A}_{t}(f,g)=(1-t)f+t g$, with the following
\begin{multline*}
\frac{\sin(\pi\lambda)}{\pi}\int_0^1\frac{t^{\lambda-1}}{(1-t)^{\lambda}}(1-t)dt=\frac{\sin(\pi\lambda)}{\pi}B\big(\lambda,2-\lambda\big)
\\=\frac{\sin(\pi\lambda)}{\pi}\Gamma(\lambda)\Gamma(2-\lambda)
=\frac{\sin(\pi\lambda)}{\pi}(1-\lambda)\Gamma(\lambda)\Gamma(1-\lambda)=1-\lambda
\end{multline*}
and (by similar arguments)
$$\frac{\sin(\pi\lambda)}{\pi}\int_0^1\frac{t^{\lambda-1}}{(1-t)^{\lambda}}t dt=\lambda,$$
we deduce that
\begin{equation*}\label{23}
{\mathcal A}_{\lambda}(f,g)-{\mathcal G}_{\lambda}(f,g)=\frac{\sin(\pi\lambda)}{\pi}\int_0^1\frac{t^{\lambda-1}}{(1-t)^{\lambda}}
\Big({\mathcal A}_{t}(f,g)-{\mathcal H}_{t}(f,g)\Big)dt,
\end{equation*}
which is the desired result.
\end{proof}

\begin{proposition}
Let $(f,g)\in{\mathcal D}(H)$ and $\lambda\in(0,1)$. Then the following equality
\begin{multline}\label{24}
{\mathcal A}(f,g)(x)-{\mathcal{HZ}}_{\lambda}(f,g)(x)\\
=\frac{\sin(\pi\lambda)}{\pi}\int_0^1HZ_{\lambda}\Big(\frac{1}{t},\frac{1}{1-t}\Big)\Big({\mathcal A}_t(f,g)(x)-{\mathcal H}_t(f,g)(x)\Big)dt
\end{multline}
holds for any $x\in H$.
\end{proposition}
\begin{proof}
If $x\notin{\rm dom}\;f\cap{\rm dom}\;g\neq\emptyset$, i.e. $f(x)=+\infty$ or $g(x)=+\infty$, then the two sides of \eqref{24} are infinite and so \eqref{24} holds. Assume that $x\in{\rm dom}\;f\cap{\rm dom}\;g\neq\emptyset$. First, it is easy to check that
\begin{equation}\label{25}
{\mathcal A}_{\lambda}(f,g)(x)+{\mathcal A}_{1-\lambda}(f,g)(x)=(f+g)(x):=2 {\mathcal A}(f,g)(x).
\end{equation}
Otherwise, from \eqref{22} we can write
\begin{multline}\label{26}
{\mathcal A}_{1-\lambda}(f,g)(x)-{\mathcal G}_{1-\lambda}(f,g)(x)\\
=\frac{\sin(\pi\lambda)}{\pi}\int_0^1\frac{1}{t\sharp_{1-\lambda}(1-t)}\Big({\mathcal A}_t(f,g)(x)-{\mathcal H}_t(f,g)(x)\Big)dt.
\end{multline}
Adding side to side \eqref{22} and \eqref{26}, with the help of \eqref{16} and \eqref{25}, we deduce
\begin{multline*}
{\mathcal A}(f,g)(x)-{\mathcal{HZ}}_{\lambda}(f,g)(x)\\
=\frac{\sin(\pi\lambda)}{2\pi}\int_0^1\Big(\frac{1}{t\sharp_{\lambda}(1-t)}+
\frac{1}{t\sharp_{1-\lambda}(1-t)}\Big)\Big({\mathcal A}_t(f,g)(x)-{\mathcal H}_t(f,g)(x)\Big)dt,
\end{multline*}
from which the desired result follows after a simple manipulation.
\end{proof}

\begin{remark}
For the sake of clearness for the reader, we mention that according to \eqref{145} and \eqref{17} we immediately deduce that the operator versions of \eqref{22}, \eqref{225} and \eqref{24} are, respectively, given by
$$A\nabla_{\lambda}B-A\sharp_{\lambda}B=\frac{\sin(\pi\lambda)}{\pi}\int_0^1\frac{1}{t\sharp_{\lambda}(1-t)}\Big(A\nabla_tB-A!_tB\Big)dt,$$
$$A\nabla B-A\sharp B=\frac{1}{\pi}\int_0^1\frac{1}{\sqrt{t(1-t)}}\Big(A\nabla_tB-A!_tB\Big)dt,$$
$$A\nabla B-HZ_{\lambda}(A,B)=\frac{\sin(\pi\lambda)}{\pi}\int_0^1HZ_{\lambda}\Big(\frac{1}{t},\frac{1}{1-t}\Big)\Big(A\nabla_tB-A!_tB\Big)dt.$$
\end{remark}

We now are in a position to state our second main result as recited in the following.

\begin{theorem}\label{th3}
Let $(f,g)\in{\mathcal D}(H)$ and $\lambda\in(0,1)$. Then the following inequalities hold:
\begin{equation}\label{27}
{\mathcal{HZ}}_{\lambda}(f,g)
\leq{\mathcal K}_{\theta(\lambda)}(f,g)\leq{\mathcal A}(f,g),
\end{equation}
where we set $\theta(\lambda)=1-\sin(\pi\lambda)$ for any $\lambda\in(0,1)$.
\end{theorem}
\begin{proof}
First, the right inequality of \eqref{27} follows from the right inequality in \eqref{18}. Note that, by \eqref{14}, we have ${\mathcal A}_t(f,g)-{\mathcal H}_t(f,g)\geq0$ for any $(f,g)\in{\mathcal D}(H)$ and $t\in(0,1)$. Now, by \eqref{24} with the left inequality of \eqref{4b}, we can write
$${\mathcal A}(f,g)-{\mathcal{HZ}}_{\lambda}(f,g)
\geq\frac{\sin(\pi\lambda)}{\pi}\int_0^1\frac{1}{\sqrt{t(1-t)}}\Big({\mathcal A}_t(f,g)-{\mathcal H}_t(f,g)\Big)dt.$$
Thanks to \eqref{225} we then obtain
\begin{equation}\label{28}
{\mathcal A}(f,g)-{\mathcal{HZ}}_{\lambda}(f,g)\geq\big(\sin(\pi\lambda)\big)\Big({\mathcal A}(f,g)-{\mathcal G}(f,g)\Big).
\end{equation}
To finish the proof, we have to take into account some precautions (see Remark \ref{rem1}). At $x\in H$ such that ${\mathcal K}_{\theta(\lambda)}(f,g)(x)=+\infty$, the left inequality of \eqref{27} is obviously satisfied at $x$. Now, let $x\in H$ be such that ${\mathcal K}_{\theta(\lambda)}(f,g)(x)<+\infty$. By the definition of the Heron functional mean \eqref{15} we deduce that ${\mathcal A}(f,g)(x)<+\infty$ and ${\mathcal G}(f,g)(x)<+\infty$. With this, \eqref{28} yields
$${\mathcal K}_{\theta(\lambda)}(f,g)(x)-{\mathcal{HZ}}_{\lambda}(f,g)(x)\geq0,$$
which, with Remark \ref{rem1} again, means that the left inequality of \eqref{27} is also satisfied at $x$. The proof is completed.
\end{proof}

To give more main results in the sequel, we need some lemmas. The first is recited in the following.

\begin{lemma} \cite{R}. Let $(f,g)\in{\mathcal D}(H)$ and $\lambda\in(0,1)$. Then we have
\begin{multline}\label{30}
2r_{\lambda}\Big({\mathcal A}(f,g)-{\mathcal H}(f,g)\Big)\leq {\mathcal A}_{\lambda}(f,g)-{\mathcal H}_{\lambda}(f,g)\\
\leq 2(1-r_{\lambda})\Big({\mathcal A}(f,g)-{\mathcal H}(f,g)\Big),
\end{multline}
where, as already pointed before, $r_{\lambda}:=\min(\lambda,1-\lambda)$.
\end{lemma}

We also need the following lemma that concerns a refinement of the so-called Hermite-Hadamard inequality, see \cite{DR,FA} for instance.

\begin{lemma}\label{lem4}
Let $a,b$ with $a<b$ and $\Phi:[a,b]\longrightarrow{\mathbb R}$ be a convex function. Then, for all $p\in[0,1]$ we have
\begin{equation}\label{301}
\Phi\Big(\frac{a+b}{2}\Big)\leq m(p)\leq\frac{1}{b-a}\int_a^b\Phi(t)dt\leq M(p)\leq\frac{\Phi(a)+\Phi(b)}{2},
\end{equation}
where
$$m(p):=p\Phi\Big(\frac{pb+(2-p)a}{2}\Big)+(1-p)\Phi\Big(\frac{(1+p)b+(1-p)a}{2}\Big)$$
and
$$M(p):=\frac{1}{2}\Big(\Phi\big(pb+(1-p)a\big)+p\Phi(a)+(1-p)\Phi(b)\Big).$$
If $\Phi$ is concave then the inequalities \eqref{301} are reversed.
\end{lemma}

\begin{lemma}\label{lem5}
For $\lambda\in(0,1)$ fixed, let $\Psi_{\lambda}$ be the function defined by
$$\Psi_{\lambda}(t)=\Big(\frac{t}{1-t}\Big)^{\lambda},\;\; t\in[0,1/2].$$
Then $\Psi_{\lambda}$ is concave on $[0,\frac{1-\lambda}{2}]$ and convex on $[\frac{1-\lambda}{2},1/2]$.
\end{lemma}
\begin{proof}
Simple computations lead to, for any $t\in(0,1/2]$,
$$\Psi_{\lambda}^{\prime}(t)=\lambda\frac{t^{\lambda-1}}{(1-t)^{\lambda+1}}$$
and
$$\Psi_{\lambda}^{\prime\prime}(t)=\lambda\frac{t^{\lambda-2}}{(1-t)^{\lambda+2}}(2t+\lambda-1).$$
The desired results follow.
\end{proof}

For the sake of simplicity, we introduce more notation. For $(f,g)\in{\mathcal D}(H)$ and $\lambda\in[0,1]$ we set
\begin{equation}\label{305}
{\mathcal L}_{\lambda}(f,g):=(1-\lambda){\mathcal H}(f,g)+\lambda{\mathcal A}(f,g).
\end{equation}
It is not hard to see that the following inequalities
\begin{equation}\label{307}
{\mathcal H}(f,g)\leq{\mathcal L}_{\lambda}(f,g)\leq{\mathcal K}_{\lambda}(f,g)\leq{\mathcal A}(f,g)
\end{equation}
hold for any $(f,g)\in{\mathcal D}(H)$ and $\lambda\in[0,1]$. With this, another main result is recited in what follows.

\begin{theorem}\label{th2}
Let $(f,g)\in{\mathcal D}(H)$ and $\lambda\in[0,1]$. Then, for any $p\in[0,1]$, there hold:
\begin{equation}\label{31}
{\mathcal{HZ}}_{\lambda}(f,g)\leq{\mathcal L}_{\gamma_p(\lambda)}(f,g)\leq{\mathcal K}_{\gamma_p(\lambda)}(f,g)\leq{\mathcal A}(f,g),
\end{equation}
where we set
$$\gamma_p(\lambda)=1-\frac{2\sin(\pi\lambda)}{\pi}\Big((1-\lambda){\mathcal M}_{\lambda}(p)+\lambda{\mathcal M}_{1-\lambda}(p)\Big),$$
with
$${\mathcal M}_{\lambda}(p):=\frac{M_{\lambda}(p)+m_{1-\lambda}(p)}{2},$$
$$M_{\lambda}(p):=\frac{1}{2}\left(\Big(\frac{p(1-\lambda)}{2-p(1-\lambda)}\Big)^{\lambda}+(1-p)\Big(\frac{1-\lambda}{1+\lambda}\Big)^{\lambda}\right),$$
and
$$m_{\lambda}(p):=p\left(\frac{2-2\lambda+p\lambda}{2+2\lambda-p\lambda}\right)^{\lambda}+(1-p)
\left(\frac{2-\lambda+p\lambda}{2+\lambda-p\lambda}\right)^{\lambda}.$$
\end{theorem}
\begin{proof}
For $\lambda=0$ or $\lambda=1$, it is clear that \eqref{31} are reduced to equalities. Assume that $\lambda\in(0,1)$. By \eqref{24}, with the left inequality of \eqref{30}, we have
\begin{equation}\label{32}
{\mathcal A}(f,g)-{\mathcal{HZ}}_{\lambda}(f,g)\geq\frac{2\sin(\pi\lambda)}{\pi}\int_0^1r_t HZ_{\lambda}\Big(\frac{1}{t},\frac{1}{1-t}\Big)
\Big({\mathcal A}(f,g)-{\mathcal H}(f,g)\Big)dt.
\end{equation}
Since $r_t=\min(t,1-t)$ and the function $t\longmapsto r_t HZ_{\lambda}\Big(\frac{1}{t},\frac{1}{1-t}\Big)$ is symmetric around $1/2$ then one has
\begin{multline}\label{33}
\int_0^1r_t HZ_{\lambda}\Big(\frac{1}{t},\frac{1}{1-t}\Big)dt=2\int_0^{1/2}t HZ_{\lambda}\Big(\frac{1}{t},\frac{1}{1-t}\Big)dt\\
=\int_0^{1/2}\left\{\Big(\frac{t}{1-t}\Big)^{\lambda}+\Big(\frac{t}{1-t}\Big)^{1-\lambda}\right\}dt.
\end{multline}
According to Lemma \ref{lem5}, we write
\begin{equation}\label{34}
\int_0^{1/2}\Big(\frac{t}{1-t}\Big)^{\lambda}dt:=\int_0^{1/2}\Psi_{\lambda}(t)dt=\int_0^{\frac{1-\lambda}{2}}\Psi_{\lambda}(t)dt
+\int_{\frac{1-\lambda}{2}}^{1/2}\Psi_{\lambda}(t)dt.
\end{equation}
By Lemma \ref{lem5} again, with Lemma \ref{lem4}, we have for any $p\in[0,1]$ (after some elementary computations)
$$\int_0^{\frac{1-\lambda}{2}}\Psi_{\lambda}(t)dt\geq\frac{1-\lambda}{2}M_{\lambda}(p),$$
with
\begin{multline}\label{35}
M_{\lambda}(p):=\frac{1}{2}\left(\Psi_{\lambda}\Big(p\frac{1-\lambda}{2}\Big)+(1-p)\Psi_{\lambda}\Big(\frac{1-\lambda}{2}\Big)\right)\\
=\frac{1}{2}\left(\Big(\frac{p(1-\lambda)}{2-p(1-\lambda)}\Big)^{\lambda}+(1-p)\Big(\frac{1-\lambda}{1+\lambda}\Big)^{\lambda}\right)
\end{multline}
and
$$\int_{\frac{1-\lambda}{2}}^{1/2}\Psi_{\lambda}(t)dt\geq\Big(\frac{1}{2}-\frac{1-\lambda}{2}\Big)\;l_{\lambda}(p)=\frac{\lambda}{2}\;m_{\lambda}(p),$$
with
\begin{multline}\label{36}
m_{\lambda}(p):=p\Psi_{\lambda}\left(\frac{p+(2-p)(1-\lambda)}{4}\right)+(1-p)\Psi_{\lambda}\left(\frac{1+p+(1-p)(1-\lambda)}{4}\right)\\
=p\left(\frac{2-2\lambda+p\lambda}{2+2\lambda-p\lambda}\right)^{\lambda}+(1-p)\left(\frac{2-\lambda+p\lambda}{2+\lambda-p\lambda}\right)^{\lambda}.
\end{multline}
With this, \eqref{34} yields
$$\int_0^{1/2}\Big(\frac{t}{1-t}\Big)^{\lambda}dt\geq\frac{1-\lambda}{2}M_{\lambda}(p)+\frac{\lambda}{2}\;m_{\lambda}(p),$$
and so
$$\int_0^{1/2}\Big(\frac{t}{1-t}\Big)^{1-\lambda}dt\geq\frac{\lambda}{2}M_{1-\lambda}(p)+\frac{1-\lambda}{2}\;m_{1-\lambda}(p).$$
These, with \eqref{33}, imply that (after simples manipulations)
$$\int_0^1r_t HZ_{\lambda}\Big(\frac{1}{t},\frac{1}{1-t}\Big)dt\geq(1-\lambda)\left\{\frac{M_{\lambda}(p)+m_{1-\lambda}(p)}{2}\right\}
+\lambda\left\{\frac{M_{1-\lambda}(p)+m_{\lambda}(p)}{2}\right\}$$
which, with \eqref{32}, implies that
$${\mathcal A}(f,g)-{\mathcal{HZ}}_{\lambda}(f,g)\\
\geq\frac{2\sin(\pi\lambda)}{\pi}\Big((1-\lambda){\mathcal M}_{\lambda}(p)+\lambda{\mathcal M}_{1-\lambda}(p)\Big)\Big({\mathcal A}(f,g)-{\mathcal H}(f,g)\Big).$$

With some precautions, as in the proof of Theorem \ref{th3}, we deduce the desired result.
\end{proof}

The operator versions of the previous functional results can be immediately deduced. For instance we have the following result which is the operator version of Theorem \ref{th2}.

\begin{corollary}
Let $A,B\in{\mathcal B}^{+*}(H)$ and $\lambda\in[0,1]$. Then we have
$$HZ_{\lambda}(A,B)\leq L_{\gamma_p(\lambda)}(A,B)\leq K_{\gamma_p(\lambda)}(A,B)\leq A\nabla B,$$
where $\gamma_p(\lambda)$ is the same as in Theorem \ref{th2} and $L_{\lambda}(A,B)=(1-\lambda)A!B+\lambda A\nabla B$ for any $\lambda\in(0,1)$, with $A!B:=A!_{1/2}B$.
\end{corollary}

Now, in the aim to give lower bounds of ${\mathcal{HZ}}_{\lambda}(f,g)$, we need to introduce another notation. For $(f,g)\in{\mathcal D}(H)$ and $\lambda\in[0,1]$ we set
\begin{equation}\label{37}
\Theta_{\lambda}(f,g)=\frac{1}{2}\Big({\mathcal H}_{\lambda}(f,g)+{\mathcal H}_{1-\lambda}(f,g)\Big).
\end{equation}
By virtue of the second relation of \eqref{135} we have $\Theta_{\lambda}(f,g)=\Theta_{\lambda}(g,f)$, and by \eqref{14} and \eqref{25}, we have $\Theta_{\lambda}(f,g)\leq{\mathcal A}(f,g)$, for any $\lambda\in[0,1]$.

We need to prove the following lemma.

\begin{lemma}\label{lem7}
Let $(f,g)\in{\mathcal D}(H)$. Then the following assertions hold:\\
(i) The map $t\longmapsto{\mathcal H}_t(f,g)$ is point-wisely convex in $t\in[0,1]$. That is, for all
$x\in H,\; t_1,t_2\in[0,1]$ and $p\in[0,1]$ one has
$${\mathcal H}_{(1-p)t_1+pt_2}(f,g)(x)\leq(1-p){\mathcal H}_{t_1}(f,g)(x)+p{\mathcal H}_{t_2}(f,g)(x).$$
(ii) The map $t\longmapsto{\mathcal A}_t(f,g)-{\mathcal H}_t(f,g)$ is point-wisely concave in $t\in[0,1]$.\\
(iii) The map $t\longmapsto\Theta_t(f,g)$ is also point-wisely convex in $t\in[0,1]$.
\end{lemma}
\begin{proof}
(i) By definition, we have, for all $x\in H$,
$${\mathcal H}_t(f,g)(x):=\Big((1-t)f^*+tg^*\Big)^*(x):=\sup_{x^*\in H}\Big\{\Re e\langle x^*,x\rangle-(1-t)f^*(x^*)-tg^*(x^*)\Big\}.$$
Fixing $x\in H$, we consider the following family of functionals, indexed by $x^*\in H$,
$$t\longmapsto\Re e\langle x^*,x\rangle-(1-t)f^*(x^*)-tg^*(x^*),$$
which are all affine and so convex in $t\in[0,1]$. It follows that $t\longmapsto{\mathcal H}_t(f,g)(x)$ is convex as a supremum of a family of convex functionals.\\
(ii) For any $x\in H$, the map $t\longmapsto{\mathcal A}_{t}(f,g)(x)$ is affine and so concave. This, with (i), implies the desired result.\\
(iii) Since $t\longmapsto{\mathcal H}_t(f,g)(x)$ is convex then so is $t\longmapsto{\mathcal H}_{1-t}(f,g)(x)$, because the real function $t\longmapsto1-t$ is affine. This, with \eqref{37}, implies that $t\longmapsto\Theta_t(f,g)$ is point-wisely convex in $t\in[0,1]$.
\end{proof}

\begin{proposition}
Let $f$ and $g$ be as above. Then, for any $\lambda\in[0,1]$, we have
\begin{equation}\label{38}
{\mathcal H}(f,g)\leq\Theta_{\lambda}(f,g)\leq{\mathcal{HZ}}_{\lambda}(f,g)\leq{\mathcal A}(f,g)
\end{equation}
\end{proposition}
\begin{proof}
For $\lambda=0$ or $\lambda=1$, \eqref{38} are reduced to ${\mathcal H}(f,g)\leq{\mathcal A}(f,g)$. Assume that below $\lambda\in(0,1)$. The two right inequalities of \eqref{38} follow from the left inequality in \eqref{14}, \eqref{16} and \eqref{27}. For proving the left inequality in \eqref{38} we proceed as follows, by using Lemma \ref{lem7},(i),
$$\Theta_{\lambda}(f,g):=\frac{{\mathcal H}_{\lambda}(f,g)+{\mathcal H}_{1-\lambda}(f,g)}{2}\geq
{\mathcal H}_{\frac{\lambda+(1-\lambda)}{2}}(f,g)={\mathcal H}_{1/2}(f,g):={\mathcal H}(f,g).$$
The proof is complete.
\end{proof}

\begin{proposition}\label{pr3}
Let $(f,g)\in{\mathcal D}(H)$. Then the map
$$t\longmapsto\frac{{\mathcal A}_t(f,g)-{\mathcal H}_t(f,g)}{t(1-t)}$$
is point-wisely integrable on $(0,1)$. That is, for any $x\in H$, the integral
\begin{equation}\label{40}
{\mathcal J}(f,g)(x):=\int_0^1\frac{{\mathcal A}_t(f,g)(x)-{\mathcal H}_t(f,g)(x)}{t(1-t)}dt
\end{equation}
exists in ${\mathbb R}\cup\{+\infty\}$. Furthermore, we have
\begin{equation}\label{41a}
{\mathcal J}(f,g)(x)=2\;\int_0^{1/2}\frac{{\mathcal A}(f,g)(x)-\Theta_t(f,g)(x)}{t(1-t)}dt.
\end{equation}
\end{proposition}
\begin{proof}
Because the map $t\longmapsto{\mathcal A}_t(f,g)-{\mathcal H}_t(f,g)$ is point-wisely concave in $t\in[0,1]$ then the map
$$t\longmapsto\frac{{\mathcal A}_t(f,g)-{\mathcal H}_t(f,g)}{t(1-t)}$$
is point-wisely continuous on $(0,1)$ and so its point-wise integral over $(0,1)$ exists in ${\mathbb R}\cup\{+\infty\}$. It is easy to see that (by using the change of variables $t=1-s$)
\begin{equation}\label{41b}
{\mathcal J}(f,g)(x):=\int_0^1\frac{{\mathcal A}_{1-t}(f,g)(x)-{\mathcal H}_{1-t}(f,g)(x)}{t(1-t)}dt
\end{equation}
Adding side to side \eqref{40} and \eqref{41b} we obtain, with the help of \eqref{25},
\begin{equation}\label{41c}
{\mathcal J}(f,g)(x):=\int_0^1\frac{{\mathcal A}(f,g)(x)-\Theta_{t}(f,g)(x)}{t(1-t)}dt.
\end{equation}
The map
$$t\longmapsto\frac{{\mathcal A}(f,g)(x)-\Theta_{t}(f,g)(x)}{t(1-t)}$$
is symmetric around $1/2$, for any $x\in H$, we then deduce \eqref{41a} from \eqref{41c}, so completes the proof.
\end{proof}

If in \eqref{40} we take $f=f_A$ and $g=f_B$, with $A,B\in{\mathcal B}^{+*}(H)$ then we obtain, by using \eqref{145},
\begin{multline}\label{41d}
{\mathcal J}(f_A,f_B)=\int_0^1\frac{{\mathcal A}_t(f_A,f_B)-{\mathcal H}_t(f_A,f_B)}{t(1-t)}dt\\=\int_0^1
\frac{f_{A\nabla_tB}-f_{A!_tB}}{t(1-t)}dt=f_{J(A,B)},
\end{multline}
where $J(A,B)$ is given by
\begin{equation}\label{42}
J(A,B)=\int_0^1\frac{A\nabla_tB-A!_tB}{t(1-t)}dt.
\end{equation}
The operator integral $J(A,B)$ can be exactly computed as recited in the following result.

\begin{proposition}\label{pr5}
With the above, we have
\begin{equation}\label{43}
J(A,B)=(B-A)A^{-1}S(A|B),
\end{equation}
where $S(A|B)$ refers to the relative operator entropy given by
$$S(A|B):=A^{1/2}\log\Big(A^{-1/2}BA^{-1/2}\Big)A^{1/2}.$$
\end{proposition}
\begin{proof}
By Kubo-Ando theory \cite{KA}, we first compute the integral $J(A,B)$ for scalar case. That is, we need to compute the following real integral
$$J(a,b)=\int_0^1\frac{a\nabla_tb-a!_tb}{t(1-t)}dt:=\int_0^1\frac{(1-t)a+tb-\big((1-t)(1/a)+t(1/b)\big)^{-1}}{t(1-t)},$$
for any real numbers $a,b>0$. Simple computation leads to (after all reductions)
$$J(a,b)=\int_0^1\frac{(a-b)^2}{ta+(1-t)b}dt=(a-b)\big(\log a-\log b\big).$$
It follows that, for any $T\in{\mathcal B}^{+*}(H)$ we have
\begin{equation}\label{44}
J(I,T)=(T-I)\log T,
\end{equation}
where $I$ denotes the identity operator of ${\mathcal B}^{+*}(H)$. Since $(A,B)\longmapsto A\nabla_tB$ and $(A,B)\longmapsto A!_tB$ are operator means in the Kubo-Ando sense then we can write, by using \eqref{44} with $T=A^{-1/2}BA^{-1/2}$,
\begin{multline*}
J(A,B)=A^{1/2}J\Big(I,A^{-1/2}BA^{-1/2}\Big)A^{1/2}\\
=A^{1/2}\Big(A^{-1/2}BA^{-1/2}-I\Big)\log\Big(A^{-1/2}BA^{-1/2}\Big)A^{1/2}.
\end{multline*}
The desired result follows after a simple manipulation.
\end{proof}

\begin{remark}
From \eqref{42}, it is immediate that $J(A,B)$ is a positive operator, since $A\nabla_tB\geq A!_tB$. Further, $J(A,B)$ is symmetric in $A$ and $B$. These properties of $J(A,B)$ are not easy to deduce from \eqref{43}.
\end{remark}

Now, we are in a position to state the following main result which gives a lower bound of ${\mathcal{HZ}}_{\lambda}(f,g)$.

\begin{theorem}\label{th5}
Let $(f,g)\in{\mathcal D}(H)$ and $\lambda\in[0,1]$. Then we have
\begin{equation}\label{50}
{\mathcal{HZ}}_{\lambda}(f,g)\\
\geq{\mathcal K}_{\delta(\lambda)}(f,g)-\alpha(\lambda)\frac{\sin(\pi\lambda)}{2\pi}{\mathcal J}(f,g),
\end{equation}
where ${\mathcal J}(f,g)$ was defined in Proposition \ref{pr3} and
$$\alpha(\lambda)=(2\lambda-1)^2,\;\; \delta(\lambda)=1-4\lambda(1-\lambda)\sin(\pi\lambda).$$
\end{theorem}
\begin{proof}
For $\lambda=0$ or $\lambda=1$, \eqref{50} are immediate after a simple discussion as in the proof of Theorem \ref{th3}. We now assume that $\lambda\in(0,1)$. By \eqref{24}, with \eqref{5}, we have
\begin{multline}\label{52}
{\mathcal A}(f,g)-{\mathcal{HZ}}_{\lambda}(f,g)\\
\leq\frac{\sin(\pi\lambda)}{\pi}\int_0^1K_{\alpha(\lambda)}\Big(\frac{1}{t},\frac{1}{1-t}\Big)
\Big({\mathcal A}_t(f,g)-{\mathcal H}_t(f,g)\Big)dt,
\end{multline}
where $\alpha(\lambda)=(2\lambda-1)^2$. By \eqref{3a} we have
$$K_{\alpha(\lambda)}=\big(1-\alpha(\lambda)\big)\frac{1}{\sqrt{t(1-t)}}+\frac{\alpha(\lambda)}{2}\frac{1}{t(1-t)}.$$
With this \eqref{52} yields
\begin{multline*}
{\mathcal A}(f,g)-{\mathcal{HZ}}_{\lambda}(f,g)\leq4\lambda(1-\lambda)\frac{\sin(\pi\lambda)}{\pi}
\int_0^1\frac{1}{\sqrt{t(1-t)}}\Big({\mathcal A}_t(f,g)-{\mathcal H}_t(f,g)\Big)dt\\
+\frac{\alpha(\lambda)}{2}\frac{\sin(\pi\lambda)}{\pi}
\int_0^1\frac{{\mathcal A}_t(f,g)-{\mathcal H}_t(f,g)}{t(1-t)}dt.
\end{multline*}
According to \eqref{225}, with \eqref{40}, we obtain
$${\mathcal A}(f,g)-{\mathcal{HZ}}_{\lambda}(f,g)\leq4\lambda(1-\lambda)\sin(\pi\lambda)\Big({\mathcal A}(f,g)-{\mathcal G}(f,g)\Big)+
\frac{\alpha(\lambda)}{2}\frac{\sin(\pi\lambda)}{\pi}{\mathcal J}(f,g).$$
We then deduce our desired inequality \eqref{50}, after a simple discussion, as in the proof of Theorem \ref{th3}. Details are routine and therefore omitted here.
\end{proof}

The operator version of Theorem \ref{th5} reads as follows.

\begin{corollary}
Let $A,B\in{\mathcal B}^{+*}(H)$ and $\lambda\in[0,1]$. Then there holds:
\begin{equation}\label{53}
HZ_{\lambda}(A,B)\geq K_{\delta(\lambda)}(A,B)-\alpha(\lambda)\frac{\sin(\pi\lambda)}{2\pi}J(A,B),
\end{equation}
where $\alpha(\lambda)$ and $\delta(\lambda)$ are as in Theorem \ref{th5} and $J(A,B)$ is as in Proposition \ref{pr5}.
In particular, for any $a,b>0$ and $\lambda\in[0,1]$, we have
\begin{multline}\label{54}
HZ_{\lambda}(a,b)\geq\frac{a+b}{2}-2\lambda(1-\lambda)\sin(\pi\lambda)\big(\sqrt{a}-\sqrt{b}\big)^2\\
-(2\lambda-1)^2\;\frac{\sin(\pi\lambda)}{2\pi}(b-a)\log(b/a).
\end{multline}
\end{corollary}
\begin{proof}
Taking $f=f_A$ and $g=f_B$ in \eqref{50}, with the help of \eqref{17} and \eqref{41d}, we obtain \eqref{53}. We then deduce \eqref{54} after simple manipulations. Details are simple and therefore omitted here.
\end{proof}

We now present some comments about comparison between \eqref{5a}, \eqref{5b} and \eqref{54}. First, we mention that, for $\lambda=0$ or $\lambda=1$, \eqref{5a} and \eqref{54}, which are both reduced to an equality, imply \eqref{5b}. For $\lambda=1/2$, \eqref{5b} and \eqref{54} are both reduced to the same equality and yield \eqref{5a}. We now consider the general case. For this purpose, by putting $a=t^2, b=1$ we consider the following function
$$f_{\lambda}(t) = g_{\lambda}(t) (t-1) \log t$$
for $t>0$ and $0\leq \lambda \leq 1$, where
\begin{equation}\label{55}
g_{\lambda}(t)=(t+1)\left\{\left(\lambda(1-\lambda)-\frac{(2\lambda-1)^2\sin(\pi\lambda)}{\pi}\right)
-\frac{2\lambda(1-\lambda) (t-1)\sin(\pi\lambda)}{(t+1)\log t}\right\}.
\end{equation}
We find
$$\lim_{t \to 0} g_{\lambda}(t) = h(\lambda),$$
where
$$h(\lambda) = \lambda(1-\lambda) -\frac{(2\lambda-1)^2 \sin(\pi \lambda)}{\pi}.$$
Here we show $h(\lambda) \geq 0$ for $0\leq \lambda \leq 1$.
We calculate
$$h'(\lambda) = \frac{(1-2\lambda)}{\pi} k(\lambda),$$
where
$$k(\lambda) = \pi -\pi(1-2\lambda)\cos(\pi \lambda) + 4\sin(\pi \lambda).$$
Since $k(\lambda) = k(1-\lambda)$, we show $k(\lambda) \geq 0$ for $0\leq \lambda \leq \frac{1}{2}$. Then we calculate
$$k'(\lambda) =\pi(6\cos(\pi \lambda)+\pi(1-2\lambda)\sin(\pi\lambda)) \geq 0$$
for $0<\lambda \leq \frac{1}{2}$ so that we have $k(\lambda) \geq k(0) =0$. Therefore $h'(\lambda) \geq 0$ for $0\leq \lambda \leq \frac{1}{2}$ and $h'(\lambda) \leq 0$ for $\frac{1}{2} \leq \lambda \leq 1$ and $h(0)=h(1)=0$ which implies $h(\lambda) \geq 0$ for $0 \leq \lambda \leq 1$. In addition, from \eqref{55} we find that $\lim_{t\to\infty}g_{\lambda}(t)=\infty$ by $h(\lambda) \geq 0$.
Therefore we have $f_{\lambda}(t) \geq 0$ for enough large $t$ and enough small $t$ which means the lower bound in \eqref{54} gives better bound than that in \eqref{5a}. Moreover the same results are shown by numerical computations in almost cases for $t>0$ and $0 \leq \lambda \leq 1$.
However, it is not true that $f_{\lambda}(t) \geq 0$ for all $t>0$ and $0 \leq \lambda \leq 1$ in general, since we have counter-examples such as $f_{0.9}(0.75)\simeq -0.0000722089$ and $f_{0.9}(1.5)\simeq -0.000197205$.

We also set the function, the square of R.H.S. in \eqref{54} minus R.H.S. in \eqref{5b} with $a=t$ and $b=1$ by
$
\frac{1}{4}\left(\alpha_{\lambda}(t)+\beta_{\lambda}(t)^2\right)
$
for $t>0$ and $0\leq \lambda \leq 1$. Where
$$\alpha_{\lambda}(t) = 2(1-r_{\lambda})(t-1)^2-(t+1)^2$$
and
$$\beta_{\lambda}(t) =t+1-4\lambda(1-\lambda)(\sqrt{t}-1)^2\sin(\pi \lambda)-\frac{(2\lambda -1)^2(t-1)(\log t) \sin(\pi \lambda)}{\pi}.$$
We easily find the lower bounds both in \eqref{54} and \eqref{5b} are symmetric with respect to $\lambda=\frac{1}{2}$ so that we consider the case $\lambda \in [0,\frac{1}{2}]$. We also find that $\alpha_{\lambda}(t)=2(1-\lambda)(t-1)^2-(t+1)^2$ is decreasing in $\lambda \in [0,\frac{1}{2}]$ for any $t>0$, since $r_{\lambda}=\lambda$.
Numerical computations show $\alpha_{\lambda}(t) +\beta_{\lambda}(t)^2 \geq 0$ for any $t>0$ and $\lambda \in [0,\frac{1}{2}]$ which means our lower bound in \eqref{54} is tighter than that in \eqref{5b}. However we have not found the analytical proof for the inequality $\alpha_{\lambda}(t) +\beta_{\lambda}(t)^2 \geq 0$ for any $t>0$ and $\lambda \in [0,\frac{1}{2}]$ and also not found any counter-examples for the inequality $\alpha_{\lambda}(t) +\beta_{\lambda}(t)^2 \geq 0$.\\

Finally, we state the following remark which may be of interest for the reader.

\begin{remark}
Let ${\mathcal M}(f,g)$ be one of the functional means (depending on $\lambda\in[0,1]$ or not) previously introduced. All theoretical results and inequalities investigated in this paper are still valid for any $f,g:H\longrightarrow{\mathbb R}\cup\{+\infty\}$ such that ${\rm dom}\;f\cap{\rm dom}\;g\neq\emptyset$. As the reader can remark it, the condition $f,g\in\Gamma_0(H)$ was not needed in the proofs. In fact, the condition $f,g\in\Gamma_0(H)$ is needed only when we want to ensure the axiom ${\mathcal M}(f,f)=f$ for any $f\in\Gamma_0(H)$. This because every functional mean was introduced as an extension of its related operator mean, denote it by $m(A,B)$ which, in its turn, should satisfy by definition the idempotent axiom for an operator mean, namely $m(A,A)=A$ for any $A\in{\mathcal B}^{+*}(H)$.
\end{remark}

\section*{\bf Acknowledgment}
The author (S.F.) was partially supported by JSPS KAKENHI Grant Number 16K05257.

\bibliographystyle{amsplain}

\end{document}